\documentclass[a4paper, 12pt]{amsart}

\usepackage{amsmath, amsthm, amssymb, mathtools}
\usepackage{microtype}
\usepackage{graphicx}
\usepackage[margin=1in]{geometry}
\usepackage[usenames,dvipsnames]{xcolor}
\usepackage{hyperref}
\hypersetup{
    colorlinks=true,
    linkcolor=cyan!80!black,
    citecolor=MidnightBlue,
    urlcolor=magenta,
}

\newtheorem{theorem}{Theorem}[section]
\newtheorem{proposition}[theorem]{Proposition}
\newtheorem{lemma}[theorem]{Lemma}
\newtheorem{corollary}[theorem]{Corollary}

\newtheorem{remark}{Remark}[section]

\newtheorem{conjecture}[theorem]{Conjecture}

\newcommand{\cS}{\mathcal{S}}
\newcommand{\E}{\mathbb{E}}
\renewcommand{\P}{\mathbb{P}}
\DeclareMathOperator{\var}{var}
\newcommand{\poi}{\mathsf{Poisson}}
\newcommand{\bern}{\mathsf{Bernoulli}}
\newcommand{\bino}{\mathsf{Binomial}}

\title{On a generalisation of the coupon collector problem}

\author[S. Athreya]{Siva Athreya}
\address{
    International Centre for Theoretical Sciences\\
    Survey No. 151, Shivakote, Hesaraghatta Hobli, Bengaluru 560089
    \and
    Stat-Math Unit\\
    Indian Statistical Institute\\
    8th Mile Mysore Road, Bengaluru 560059
}
\email{athreya@icts.res.in}

\author[S. Mukherjee]{Satyaki Mukherjee}
\address{
    Technical University of Munich\\
    Department of Informatics - I7\\
    Boltzmannstr. 3\\
    85748 Garching b. M\"unchen\\
    Germany
}
\email{paglasatyaki@gmail.com}

\author[S. S. Mukherjee]{Soumendu Sundar Mukherjee}
\address{
    Statistics and Mathematics Unit\\
    Indian Statistical Institute\\
    203 B. T.~Road, Kolkata 700108, India
}
\email{ssmukherjee@isical.ac.in}

\begin{document}

\maketitle

\begin{abstract}
We consider a generalisation of the classical coupon collector problem. We define a super-coupon to be any $s$-subset of a universe of $n$ coupons. In each round, a random $r$-subset from the universe is drawn and all its $s$-subsets are marked as collected. We show that the time to collect all super-coupons is $\binom{r}{s}^{-1}\binom{n}{s} \log \binom{n}{s}(1 + o(1))$ on average and has a Gumbel limit after a suitable normalisation. In a similar vein, we show that for any $\alpha \in (0, 1)$, the expected time to collect $(1 - \alpha)$ proportion of all super-coupons is $\binom{r}{s}^{-1}\binom{n}{s} \log \big(\frac{1}{\alpha}\big)(1 + o(1))$. The $r = s$ case of this model is equivalent to the classical coupon collector model.

We also consider a temporally dependent model where the $r$-subsets are drawn according to the following Markovian dynamics: the $r$-subset at round $k + 1$ is formed by replacing a random coupon from the $r$-subset drawn at round $k$ with another random coupon from outside this $r$-subset. We link the time it takes to collect all super-coupons in the $r = s$ case of this model to the cover time of random walk on a certain finite regular graph and conjecture that in general, it takes $\frac{r}{s} \binom{r}{s}^{-1}\binom{n}{s}\log\binom{n}{s}(1 + o(1))$ time on average to collect all super-coupons.
\end{abstract}

\section{Introduction}
In the classical coupon collector problem, at every unit of time a coupon is drawn randomly with replacement from a universe of $n$ distinct coupons. One stops when one has collected all the $n$ coupons. It is well known that $T$, the time it takes to collect all the coupons, is $ n \log n(1 + o(1))$ with high probability.

Various generalisations of the classical coupon collector have been studied in the literature. For instance, \cite{polya1930wahrscheinlichkeitsaufgabe} and \cite{stadje1990collector} allow drawing $q > 1$ coupons without replacement in each round. In \cite{newman1960double}, the authors consider a generalisation where one only stops when one has collected at least $m$ coupons of each type. The papers \cite{papanicolaou1998asymptotics}, \cite{holst2001extreme}, \cite{neal2008generalised} consider a variation where coupons are drawn independently of each other from the same distribution, but one which might not be uniform. Another type of generalisation where a random number of coupons are drawn in each round without replacement is considered in \cite{sellke1995many}, \cite{ivchenko1998many} and \cite{adler2001coupon}. A dependent model is considered in \cite{xu2011generalized}, where the collector draws $q$ coupons in each round, and, among these, chooses to keep the one that he has collected the least number of times so far. The authors show that on average it takes $\frac{n \log n}{q} + \frac{n}{q}(m - 1)\log \log n + O(mn)$ time for the collector to have at least $m$ coupons of each type.

In this article, we consider the following generalisation: We have a universe of $n$ coupons, denoted by $[n] := \{1, \ldots, n\}$. A \emph{super-coupon} is an $s$-subset of $[n]$, where $s$ is a fixed integer between $1$ and $n$. Let $s \le r \le n$ be an integer. In each round, $r$ distinct coupons are drawn at random from the universe. We then say that all $s$-subsets of these $r$ coupons have been collected. Thus in each round, we collect $\binom{r}{s}$ super-coupons. We want to know how long it takes to collect all the $\binom{n}{s}$ super-coupons. Let us denote this time by $T^{(r, s)}$.

Note that when $r = s$, the above problem becomes equivalent to the classical coupon collector problem. It is easy to see that $T^{(r, s)}$ is stochastically smaller than $T^{(s, s)}$, which implies that $T^{(r, s)}$ is $O(\binom{n}{s} \log \binom{n}{s}) = O(n^s \log n^s)$ with high probability. We prove (see Theorem \ref{thm:gccp}) that
\[
    \E T^{(r, s)} = \frac{\binom{n}{s}\log \binom{n}{s} }{\binom{r}{s}}(1 + o(1)) = \frac{1}{(s - 1)! \binom{r}{s}} n^s \log n (1 + o(1)).
\]
The main challenge here is to identify the constant $1/((s - 1)!\binom{r}{s})$ exactly. Our result implies that it takes asymptotically the same amount of time on average to collect all super-coupons in our model as in the simpler model where one draws $\binom{r}{s}$ distinct super-coupons at random in each round (the difference being that these $\binom{r}{s}$ super-coupons do not need to be subsets of the same $r$ coupons). Incidentally, our model encompasses the latter simpler model --- it is equivalent to the $s = 1$ case of our model (which is what \cite{polya1930wahrscheinlichkeitsaufgabe} and \cite{stadje1990collector} studied).

We also consider, for $\alpha \in (0, 1)$, the time $T^{(r, s)}_{\alpha}$ to collect $(1 - \alpha)$ proportion of all coupons. We show (see Theorem~\ref{thm:alpha-gccp}) that
\[
    \E T^{(r, s)}_{\alpha} = \frac{\binom{n}{s}\log \big(\frac{1}{\alpha}\big)}{\binom{r}{s}}(1 + o(1)).
\]

Our model is related to covering designs in combinatorics. An $s$-$(n, r, t)$ design is a collection of $r$-subsets of an $n$-set such that every possible $s$-subset is contained in at least $t$ of the $r$-subsets. In \cite{godbole1996random}, the authors prove that $\binom{n}{s} \log \binom{n}{s} / \binom{r}{s}$ forms a threshold for the probability that a uniformly random collection of $T$ many $r$-subsets forms an $s$-$(n, r, 1)$ design. That is when $T < (1 - o(1)) \binom{n}{s} \log \binom{n}{s} / \binom{r}{s}$, with high probability the collection is not an $s$-$(n, r, 1)$ design, and conversely, when $T > (1 + o(1)) \binom{n}{s} \log \binom{n}{s} / \binom{r}{s}$, with high probability the collection is an $s$-$(n, r, 1)$ design. Our problem is similar but subtly different. Our $T^{(r, s)}$ is a stopping time where the criterion for stopping is the collection becoming an $s$-$(n, r, 1)$ design. The techniques we use for computing the expectation of $T^{(r, s)}$ are based on the first and the second moment methods and are quite different from the techniques used in \cite{godbole1996random} which are based on Janson's correlation inequalities.

Distributional results are also known for a number of coupon collector problems. For example, for the classical coupon collector problem, \cite{erdHos1961classical} showed that the time to collect all coupons, suitably normalised, converges in distribution to a scaled Gumbel random variable. We too obtain via the moment method a Gumbel limit for an appropriately normalised $T^{(r, s)}_{\alpha}$ (see Theorem~\ref{thm:gumbel}). We remark here that in the covering design setting, it was shown in \cite{godbole1996random} that for $T = \binom{n}{s} (\log \binom{n}{s} + c + o(1)) / \binom{r}{s}$, the probability that a uniformly random collection of $T$ many $r$-subsets forms an $s$-$(n, r, 1)$ design, converges to $e^{-e^{-c}}$, the cumulative distribution function of standard Gumbel.

Finally, we consider a variant of our model where a temporal dependence structure is introduced in that the $r$-subsets are now drawn according to a Markovian dynamics: the $r$-subset at round $k + 1$ is formed by replacing a random coupon from the $r$-subset drawn at round $k$ with another random coupon from outside this $r$-subset. Based on our simulations (see Figure~\ref{fig:T-gccp-rw}), we formulate a conjecture that for this model, it takes $\frac{r}{s} \binom{r}{s}^{-1}\binom{n}{s}\log\binom{n}{s}(1 + o(1))$ time on average to collect all super-coupons (see Conjecture~\ref{conj:rw}). The $r = s$ case of this conjecture can be rephrased as a cover time question for an appropriate random walk and this interpretation leads to its solution using the theory of random walks on finite graphs (see Theorem~\ref{thm:rw_r=s}).

\section{Main results}
We now give precise statements of our main results.
\begin{theorem}\label{thm:gccp}
Let $s \le r$ be positive integers. In each round, we draw $r$ coupons at random without replacement from a universe of $n$ coupons. We mark all the $s$-subsets (super-coupons) of any $r$-subset already drawn as collected. Let $T^{(r, s)}$ denote the time it takes to collect all possible $\binom{n}{s}$ super-coupons.
We then have
\[
    \E T^{(r, s)} = \frac{\binom{n}{s} \log \binom{n}{s}}{\binom{r}{s}}(1 + o(1)).
\]
\end{theorem}

It is not hard to prove that there is no real difference between the model where in each round, the $r$ coupons are chosen without replacement and the one where those are chosen with replacement. Note that if $s = 1$ and we were to choose $r$ coupons with replacement, this would be the same as the classical coupon collector problem where one draws one coupon at a time, only with the unit of time scaled by a factor of $r$. Therefore, since the classical coupon collector takes $n \log n (1 + o(1))$ time, in this problem the collector would take $\frac{n \log n}{r} (1 + o(1))$ time. This agrees with our result. Unfortunately, the argument sketched above does not easily generalise to the case $s > 1$.

A natural example of our generalised coupon collector model in the case $r = 3, s = 2$ is as follows. Suppose the coupons are the edges of a complete graph on $n$ vertices. In each round, we choose $3$ edges (i.e. coupons) that form a triangle. In the classical coupon collector problem, we would not have had this restriction that the three edges need to form a triangle, which makes our problem harder to analyse. Nevertheless, our main result shows that the average time taken to see all the edges is the same for both problems, i.e. the dependency introduced by the above restriction does not manifest itself in a first order result such as ours.

Another natural question in the set-up of Theorem \ref{thm:gccp} is how long it takes to collect $(1 - \alpha)$-proportion of all the super-coupons, where $\alpha \in (0, 1)$. In the following theorem, we show that the collector needs $(1 + o(1))\binom{n}{s} \log(\frac{1}{\alpha}) / \binom{r}{s}$ time for this.

\begin{theorem} \label{thm:alpha-gccp} Suppose that $T^{(r, s)}_{\alpha}$ denotes the time it takes to observe at least $(1 - \alpha)\binom{n}{s}$ distinct super-coupons, where $\alpha \in (0, 1)$. Then
\[
    \E T^{(r, s)}_{\alpha} = \frac{\binom{n}{s}\log\big(\frac{1}{\alpha}\big)}{\binom{r}{s}}  (1 + o(1)).
\]
\end{theorem}

Thus, similar to the case of collecting all super-coupons, on average it takes asymptotically the same amount of time in our model to collect $(1 - \alpha)$ proportion of all super-coupons as in the model where in each round, $\binom{r}{s}$ super-coupons are drawn without the constraint of them forming an $r$-subset.

In Figure \ref{fig:T-gccp}, we look at how long it takes on average to see all possible pairs of coupons if we were to to sample $r$ coupons at a time. We plot Monte Carlo estimates of $T^{(r, 2)}$ (scaled down by its predicted expected value) versus $n$, fixing $r = 10$, in Figure~\ref{fig:T-gccp}-(a), and versus $r$, fixing $n = 200$, in Figure~\ref{fig:T-gccp}-(b). The error bars correspond to one Monte Carlo standard error. We find excellent agreement between the predicted and observed behaviours.

\begin{figure}[!t]
    \centering
    \begin{tabular}{cc}
        \includegraphics[scale = 0.28]{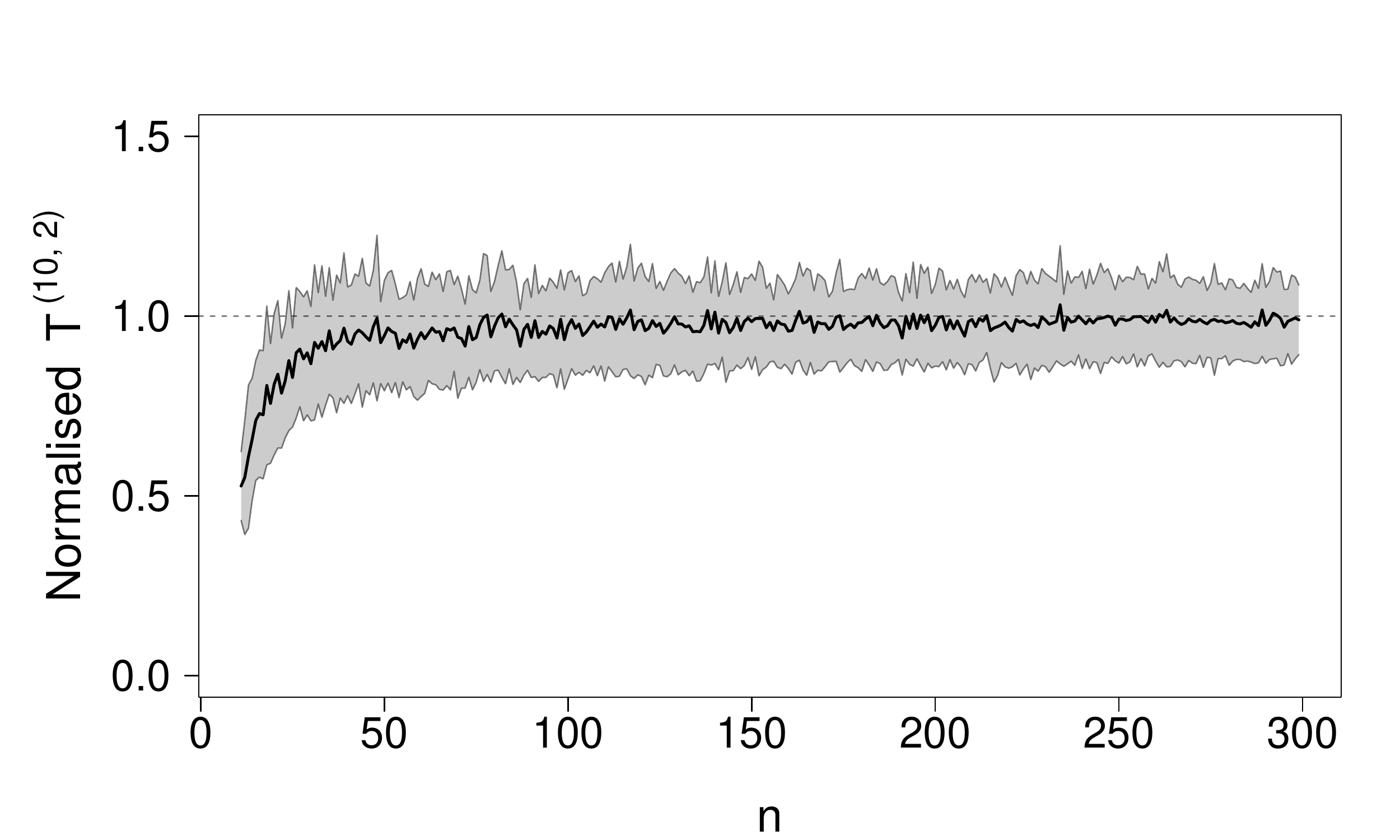} & \includegraphics[scale = 0.28]{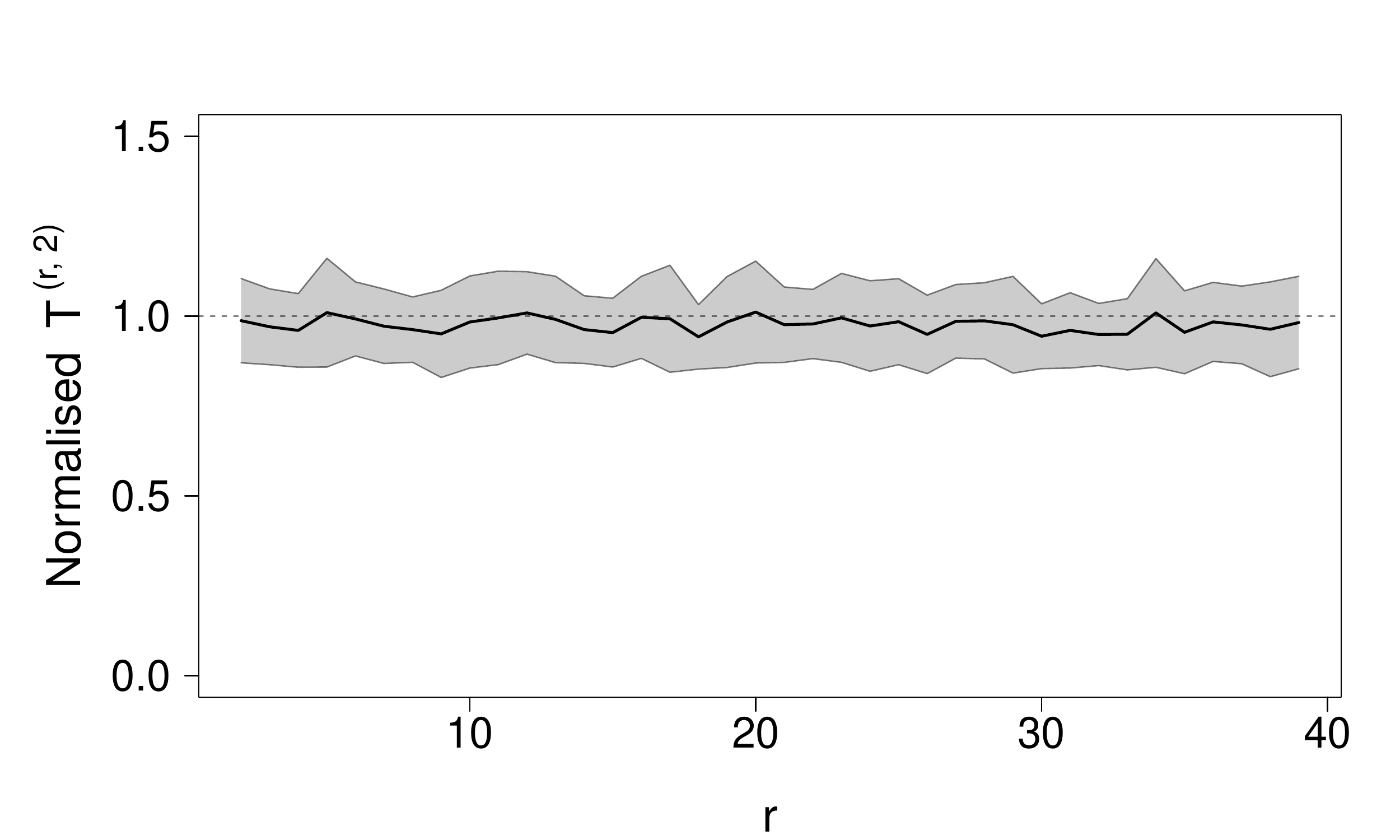} \\
        \quad\,\, (a) & \quad\,\,\,(b)
    \end{tabular}
    \caption{(a) $T^{(10, 2)}$ versus $n$; (b) $T^{(r, 2)}$ versus $r$, with $n = 200$. In both plots, $T^{(r, s)}$ is normalised by $\frac{1}{(s - 1)!\binom{r}{s}} n^s \log n$.}
    \label{fig:T-gccp}
\end{figure}

In a similar vein, Figure \ref{fig:T-alpha-gccp} demonstrates the convergence result of Theorem~\ref{thm:alpha-gccp} via two simulation experiments in the case $s = 2$. Again, we see excellent agreement with the theoretical predictions.

\begin{figure}[!t]
    \centering
    \begin{tabular}{cc}
        \includegraphics[scale = 0.28]{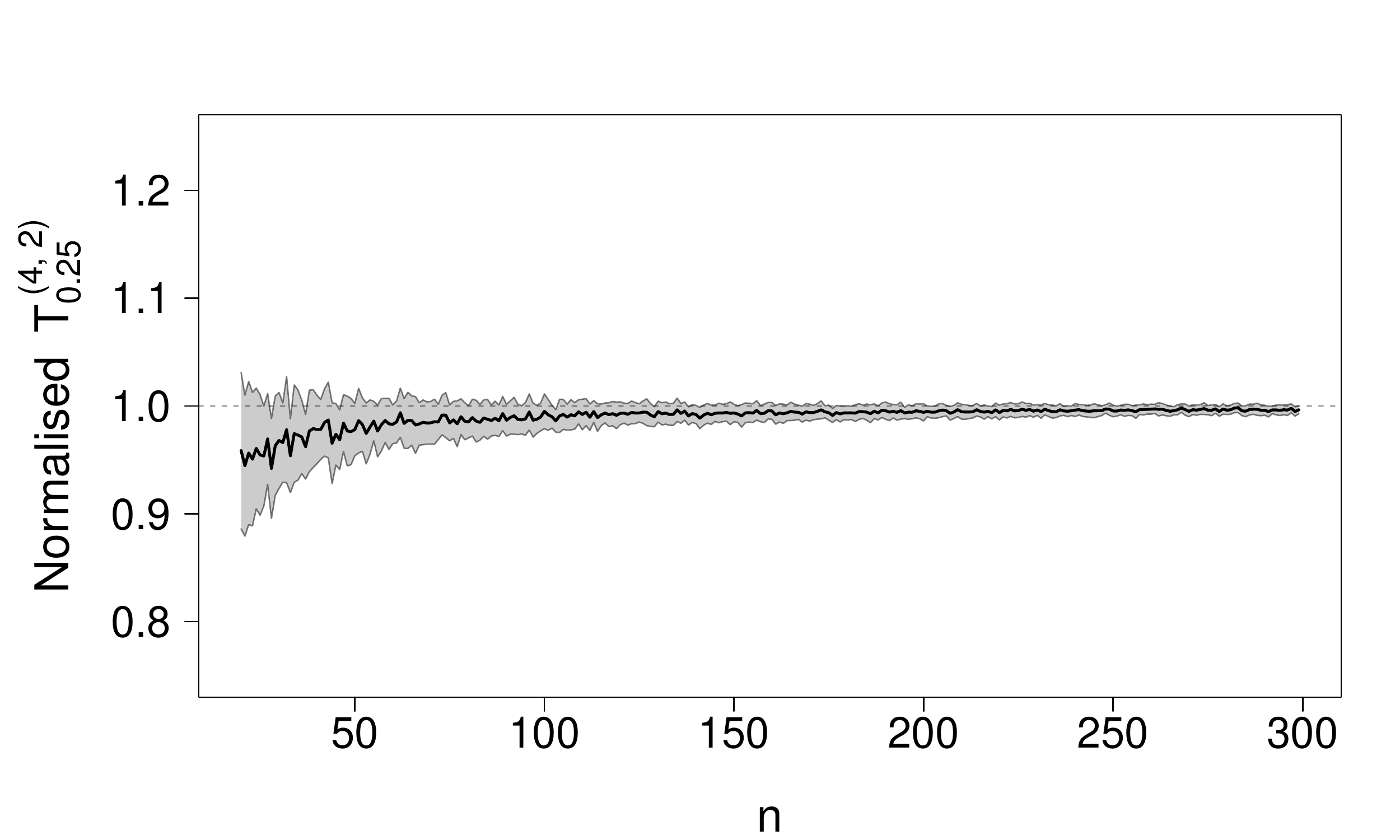} & \includegraphics[scale = 0.28]{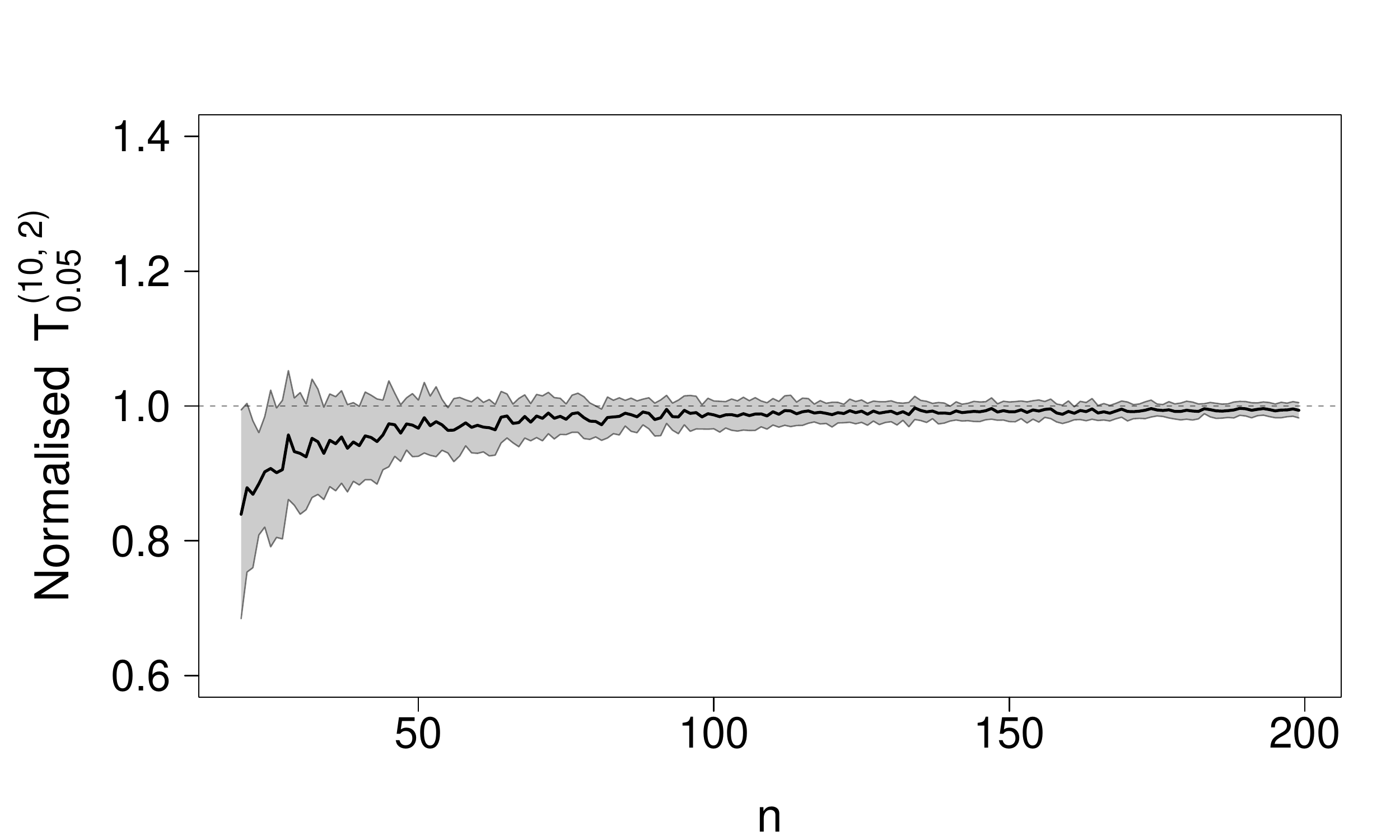} \\
        \quad\,\, (a) & \quad\,\,\,(b)
    \end{tabular}
    \caption{(a) $T^{(4, 2)}_{0.25}$ versus $n$; (b) $T^{(10, 2)}_{0.05}$ versus $n$. In both plots, $T^{(r, s)}_{\alpha}$ is normalised by $\frac{1}{s!\binom{r}{s}}n^s \log (1 / \alpha)$.}
    \label{fig:T-alpha-gccp}
\end{figure}

In the classical coupon collector problem, following an appropriate normalisation, one obtains a Gumbel limit for the time to collect all coupons. In Figure~\ref{fig:distribution}, we compare the empirical distribution of $T^{(10, 2)}$ (suitably normalised) against a standard Gumbel distribution. The excellent agreement observed in this simulation motivated us to posit and prove the following result.

\begin{figure}[!t]
    \centering
    \begin{tabular}{cc}
        \includegraphics[scale = 0.34]{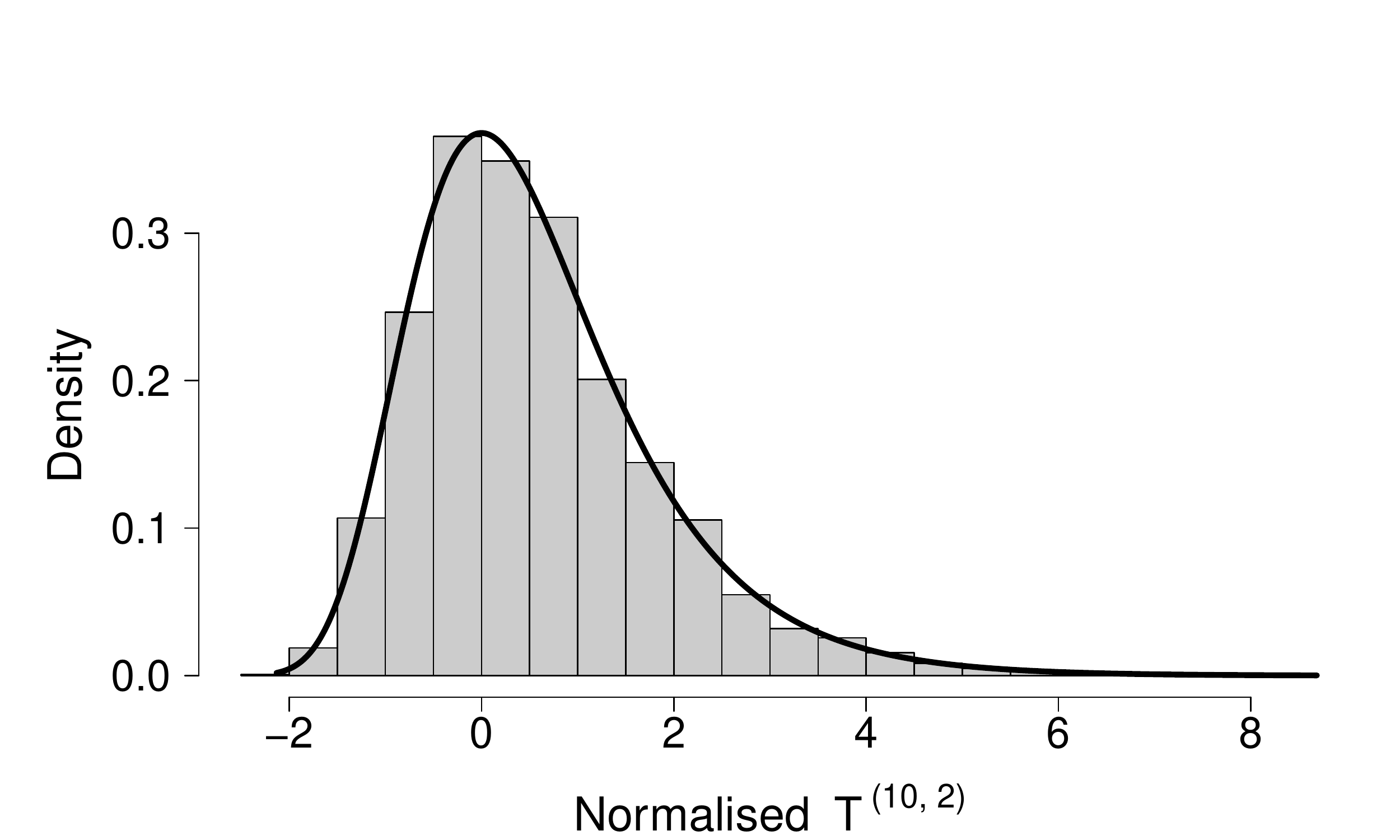} & \includegraphics[scale = 0.34]{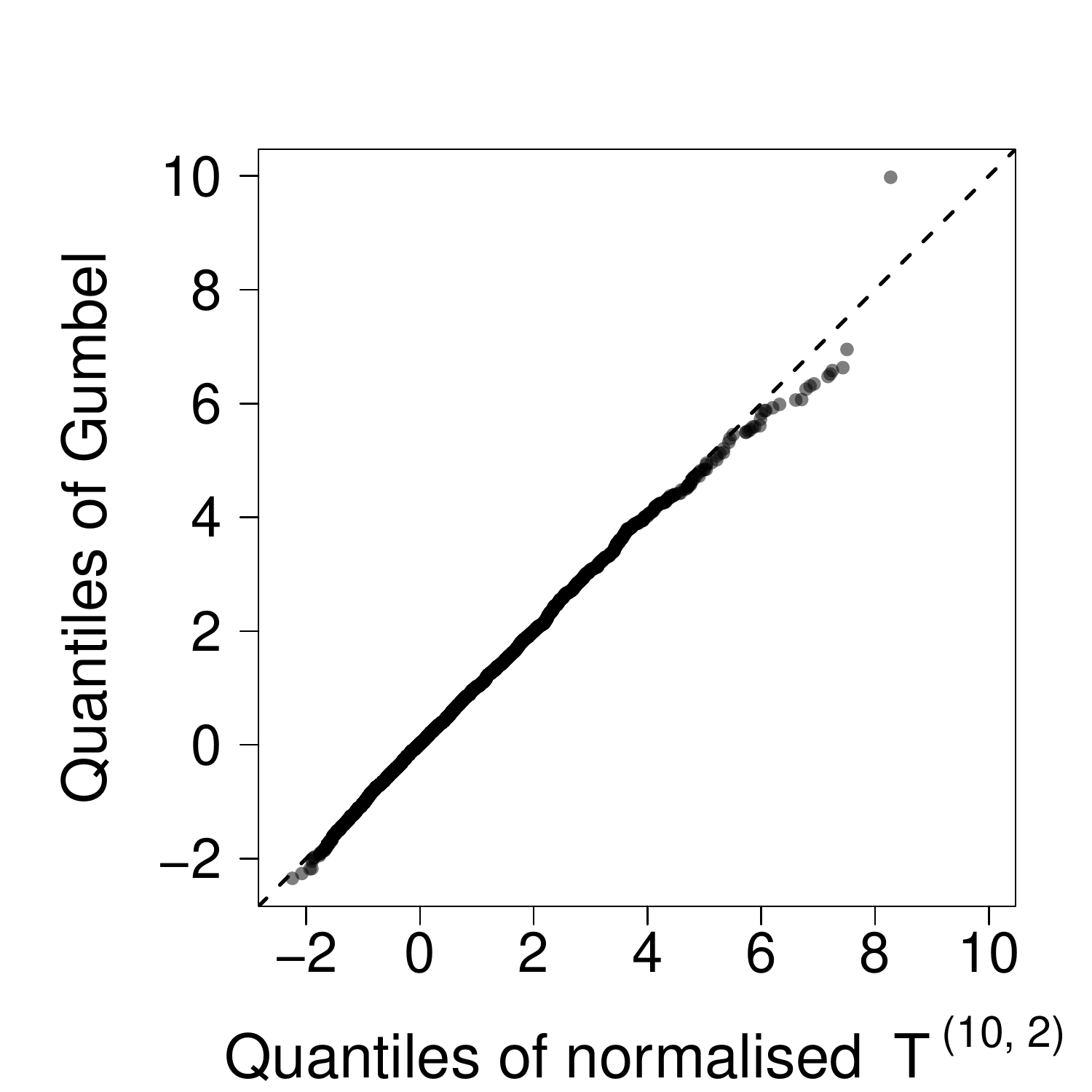} \\
        \quad\,\, (a) & \quad\,\,\,(b)
    \end{tabular}
    \caption{(a) Histogram of $T^{(10, 2)}$ with $n = 500$ over $5000$ replications (we normalise $T^{(r, s)}$ as $[\binom{n}{s} / \binom{r}{s}]^{-1}[T^{(r,s)} - \binom{n}{s} \log \binom{n}{s} / \binom{r}{s}]$). The solid curve depicts the density of a standard Gumbel distribution. (b) Q-Q plot of normalised $T^{(10, 2)}$ and standard Gumbel.}
    \label{fig:distribution}
\end{figure}

\begin{theorem}\label{thm:gumbel}
    As $n \to \infty$, $\frac{T^{(r,s)} - \binom{n}{s} \log \binom{n}{s} / \binom{r}{s}}{\binom{n}{s} / \binom{r}{s}}$ converges weakly to a standard Gumbel distribution.
\end{theorem}

Finally, we consider a further generalisation of our model. Let $H_t$ denote the $r$-subset drawn in round $t$. Note that in our problem, the $H_t$'s are independently and identically distributed as uniform over all possible $r$-subsets of the universe. A natural next step would be to introduce a temporal dependence across the $H_t$'s. The following conjecture predicts the expected time to collect all $s$-subsets of coupons (i.e. $s$-super-coupons) where $H_t$'s form a certain natural Markov chain.

\begin{conjecture}\label{conj:rw}
Perform a random walk on the space of all possible $r$-subsets of the universe of $n$ coupons as follows. Let $H_t$ denote the $r$-subset drawn in round $t$. Given $H_t$, one obtains $H_{t + 1}$ by choosing a uniformly random coupon in $H_t$ and replacing it with another uniformly chosen random coupon from outside $H_t$. Let $T^{(r,s)}_{\mathrm{RW}}$ denotes the time it takes to collect all $s$-super-coupons. Then
\[
    \E T^{(r,s)}_{\mathrm{RW}} = \frac{\frac{r}{s}\binom{n}{s} \log \binom{n}{s}}{\binom{r}{s}}(1 + o(1)).
\]
\end{conjecture}

We present simulations in support of the above conjecture in Figure \ref{fig:T-gccp-rw}.
\begin{figure}[!t]
    \centering
    \begin{tabular}{cc}
        \includegraphics[scale = 0.28]{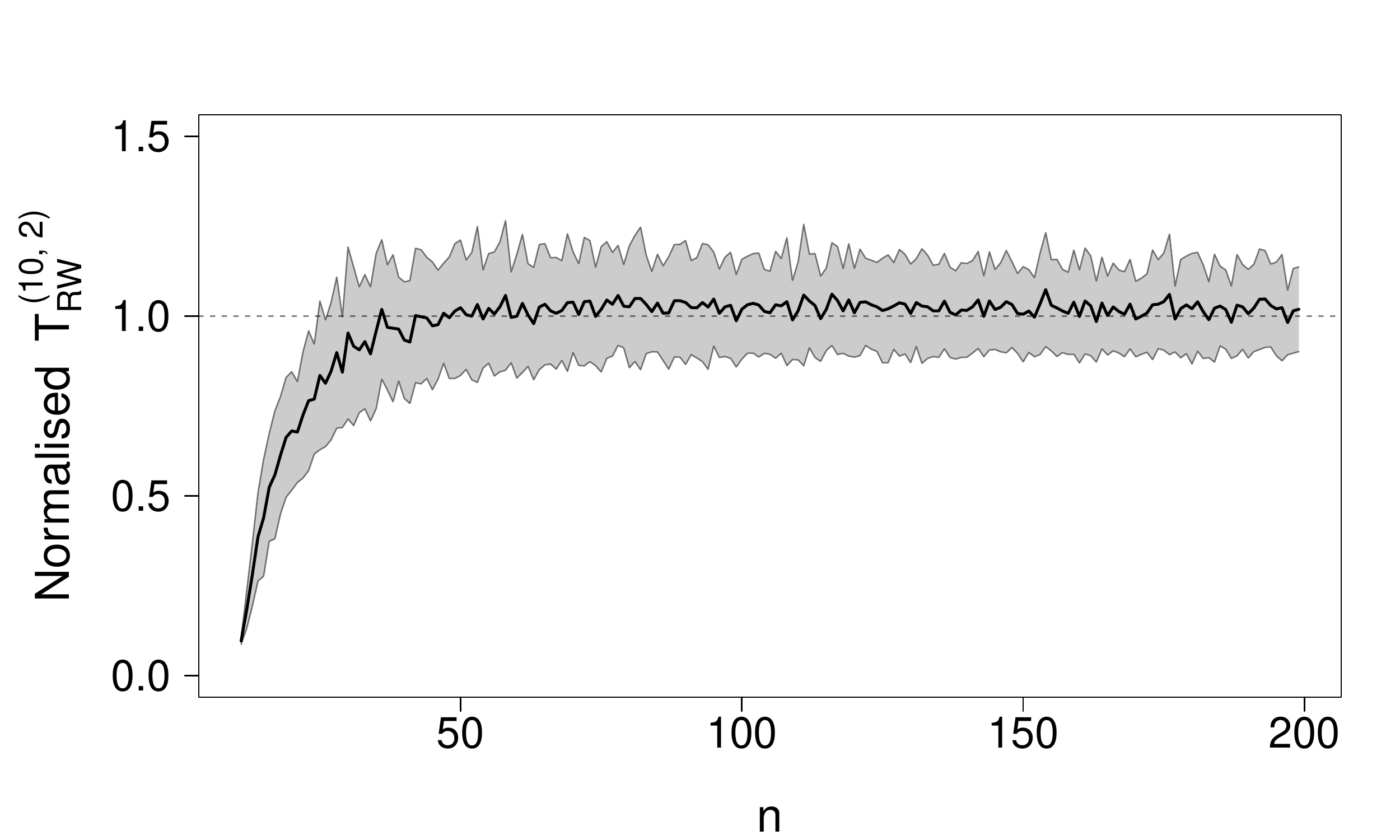} & \includegraphics[scale = 0.28]{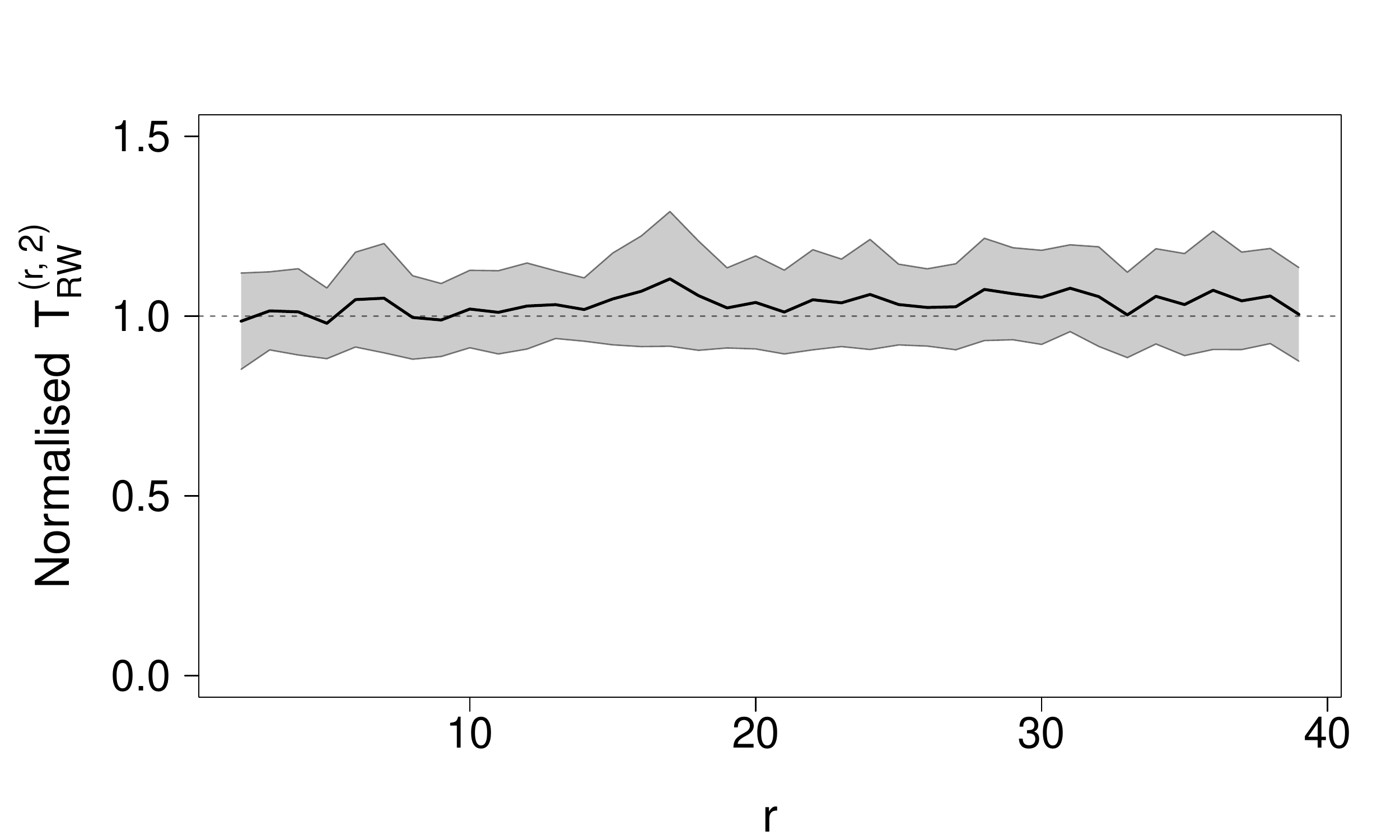} \\
        \quad\,\, (a) & \quad\,\,\,(b)
    \end{tabular}
    \caption{(a) $T^{(10, 2)}_{\mathrm{RW}}$ versus $n$; (b) $T^{(r, 2)}_{\mathrm{RW}}$ versus $r$ with $n = 200$. In both plots, $T^{(r, 2)}_{RW}$ is normalised by $\frac{1}{r - 1} n^2 \log n$.}
    \label{fig:T-gccp-rw}
\end{figure}

The $r = s$ case of Conjecture~\ref{conj:rw} can be rephrased as a cover time question. Consider a graph where each vertex is a super-coupon of size $r$ and two vertices are connected by an edge if they share all but a single coupon. This is an $r(n - r)$-regular graph on $m = \binom{n}{r}$ vertices. It is not hard to see that $T_{\mathrm{RW}}^{(r, r)} - 1$ is precisely the cover time (i.e. the expected time to visit all the vertices) of a random walk on this graph started at a uniformly random vertex. Exploiting the symmetry of this graph, we can then obtain sharp estimates of hitting times of vertices, which in turn yield a sharp estimate of the cover time, leading to the following theorem.

\begin{theorem}\label{thm:rw_r=s}
    Let $T_{\mathrm{RW}}^{(r, s)}$ be as defined in Conjecture~\ref{conj:rw}. We have
\[
    \E T_{\mathrm{RW}}^{(r, r)} = \binom{n}{r} \log \binom{n}{r}(1 + o(1)).
\]
\end{theorem}

\begin{remark}
Notice that, according to Conjecture~\ref{conj:rw},
\[
    \frac{\E T^{(r, s)}_{\mathrm{RW}}}{\E T^{(r, s)}} = \frac{r}{s} (1 + o(1)),
\]
i.e. we have a slowdown by a factor of $r/s$ in the Markovian model. Interestingly, there is no discrepancy between the two when $r = s$. A heuristic explanation of this discrepancy (and its absence in the $r = s$ case) is that in the Markovian model, we observe $\binom{r - 1}{s - 1}$ many new super-coupons in round $t + 1$, which are different from those observed in round $t$, whereas in the original model, for any given round, we observe $\binom{r}{s}$ many new super-coupons in the next round with high probability.
\end{remark}

\section{Proofs}
Fix $1\leq s \leq r \leq n.$ Let $N_k := $ the number of super-coupons that have not been collected by time $k$. It is easy to observe that
\begin{equation} \label{eq:TtoN}
    \E T^{(r, s)}_{\alpha} = \sum_{k \ge 0} \P(T^{(r, s)}_{\alpha} > k) = \sum_{k \ge 0} \P(N_k > 0).
\end{equation}
We will use the first and second moment methods to bound $\P(N_k >0).$ The next lemma gives estimates on the moments of $N_k$ for all $k \ge 0$.

\begin{lemma}\label{lem:mom}
Let $m = \binom{n}{s}$, $\theta = \frac{\binom{n - s}{r - s}}{\binom{n}{r}}$ and $\psi = \frac{r^r n^{-(s + 2)}}{\max\{r - 2s, 0\}!}$. For $k \ge 1, d \ge 2$, we have
\begin{align} \label{eq:mom_1}
    & \E N_k = m (1 - \theta)^k; \\ \label{eq:mom_d}
    & d! \binom{m}{d} (1 - d\theta)^k \le \E N_k^d \le d! \binom{m}{d} \bigg(1 - d \theta + \binom{d}{2}\psi\bigg)^k + d^2 m^{d - 1} (1 - \theta)^k.
\end{align}
\end{lemma}

\begin{proof}
Let $\cS = \{S \subset [n] : |S| = s\}$. Let $\{S_1, \ldots, S_m\}$ be an enumeration of $\cS$. We write $N_k$ as follows:
\[
    N_k = \sum_{S \in \cS} I_S^{(k)},
\]
where $I_S^{(k)}$ denotes the indicator of the event that $S$ has not been collected by time $k$. Thus
\[
    \E N_k = \binom{n}{s} \E I_{S_1}^{(k)}.
\]
Now, denoting by $E_S$ the event that $S$ is observed in round 1, we have
\begin{align*}
    \E I_S^{(k)} &= \P(S \text{ has not been collected at by time } k) \\
    &= \left[\P(S \text{ is not collected in one round}) \right]^k \\
    &= \left[\P(E_S^c)\right]^k \\
    &= \left[1 - \P(E_S) \right]^k.
\end{align*}
We get \eqref{eq:mom_1} by noting that $\P(E_S) = \frac{\binom{n - s}{r - s}}{\binom{n}{r}} = \theta$.

Let $C_{d, l}$ be the number of words of length $d$ made up of $l$ given letters such that each letter appears at least once. Note that
\begin{align} \label{eq:cdd}
    & C_{d, d} = d! \text{ and} \\ \label{eq:cdl}
    & \sum_{l = 1}^{d - 1} \binom{m}{l} C_{d, l} \leq m^d\bigg(1 - \bigg(1 - \frac{d}{m}\bigg)^d\bigg) \le d^2 m^{d - 1},
\end{align}
where the first inequality in \eqref{eq:cdl} follows from lower bounding the probability that an $m$-sided die rolled $d$ times gives a sequence of $d$ distinct values by $\big(1 - \frac{d}{m}\big)^d$.

Now consider the following decomposition:
\begin{align}\nonumber
    \E N_k^d = \E \bigg[\sum_{S_{i_1}, \ldots, S_{i_d} \in \cS} \,\, \prod_{j = 1}^d I_{S_{i_j}}^{(k)}\bigg] &= \sum_{S_{i_1}, \ldots, S_{i_d} \in \cS} \E \prod_{j = 1}^d I_{S_{i_j}}^{(k)} \\ \nonumber
    &= \sum_{l = 1}^d \sum_{\substack{S_{i_1}, \ldots, S_{i_l} \in \cS \\ S_{i_j}\text{'s are distinct}}} C_{d, l} \, \E \prod_{j = 1}^l I_{S_{i_j}}^{(k)} \\ \label{eq:prod_formula}
    &= \sum_{l = 1}^d \binom{m}{l} C_{d, l} \, \E \bigg[\prod_{j = 1}^l I_{S_j}^{(k)}\bigg],
\end{align}
which gives
\begin{equation}\label{eq:dual-bd-dth-mom}
    \binom{m}{d} C_{d, d} \E \bigg[\prod_{j = 1}^d I_{S_j}^{(k)}\bigg] \le \E N_k^d \le \binom{m}{d} C_{d, d} \, \E \bigg[\prod_{j = 1}^d I_{S_j}^{(k)}\bigg] + \sum_{l = 1}^{d - 1} \binom{m}{l} C_{d, l} \, \E I_{S_1}^{(k)}.
\end{equation}

Let us now estimate $\E \big[\prod_{j=1}^d I_{S_j}^{(k)}\big]$. First note that
\begin{align*}
    \E \bigg[\prod_{j=1}^d I_{S_j}^{(k)}\bigg] &= \P(\text{none of $S_1, \ldots, S_d$ has been observed by time $k$}) \\
    &= \left[\P(\text{none of $S_1, \ldots, S_d$ has been observed in one round})\right]^k \\
    &= \left[\P(\cap_{j = 1}^d E_{S_j}^c)\right]^k \\
    &= \left[1 - \P(\cup_{j = 1}^d E_{S_j})\right]^k.
\end{align*}

Now by the inclusion-exclusion principle,
\begin{equation}\label{eq:iep}
    \sum_{j = 1}^d \P(E_{S_j}) - \sum_{1 \le j < j' \le d} \P(E_{S_j} \cap E_{S_{j'}}) \le \P(\cup_{j = 1}^d E_{S_j}) \le \sum_{j = 1}^d \P(E_{S_j}).
\end{equation}

We note that collecting both $S_j$ and $S_{j'}$ amounts to collecting their union, which has size $2s - \ell$, where $\ell = |S_{j} \cap S_{j'}|$. Consider $S_{j} \ne S_{j'}$ such that $|S_{j} \cap S_{j'}| = \ell$. Note that $\ell \le s - 2$. Also, since the union of $S_{j}$ and $S_{j'}$ has size $2s - \ell$, we must have that $2s - \ell \le r$, i.e. $\ell \ge \max\{0, 2s - r\}$. It follows that
\begin{align*}
    \P(E_{S_{j}} \cap E_{S_{j'}}) = \frac{\binom{n - 2s + \ell}{r - 2s + \ell}}{\binom{n}{r}} &= \frac{ \frac{n^{n - 2s + \ell}}{(r - 2s + \ell)!}}{\frac{n^r}{r^r}} \\
    &\le \frac{r^r}{\max\{r - 2s, 0\}!}n^{-2s + \ell} \\
    &\le \frac{r^r}{\max\{r - 2s, 0\}!}n^{-(s + 2)} = \psi.
\end{align*}
Using this we get from \eqref{eq:iep} that
\[
    d\theta - \binom{d}{2} \psi \le \P(\cup_{j = 1}^d E_{S_j}) \le d\theta.
\]
Hence
\begin{equation}\label{eq:prod_ind_bounds}
    (1 - d\theta)^k \le \E \bigg[\prod_{j=1}^d I_{S_j}^{(k)}\bigg] \leq \bigg(1 - d \theta + \binom{d}{2}\psi\bigg)^k.
\end{equation}
Plugging this into \eqref{eq:dual-bd-dth-mom} and using \eqref{eq:cdd} and \eqref{eq:cdl}, we get the desired estimates \eqref{eq:mom_d}.
\end{proof}

The following lemma is needed to estimate $(1 - 2\theta + \psi)^k$ for $k = O(n^s \log n)$, which will be used in the proof of Theorem~\ref{thm:gccp}.
\begin{lemma} \label{lem:asym}
Let $a_n, b_n$ be sequences converging to 0 such that $ b_n = o\big(\frac{a_n}{-\log a_n}\big)$. Let $k_n$ be a sequence such that, $k_n = O\big(\frac{1}{a_n} \log \frac{1}{a_{n}}\big)$. Then $(1 - a_n + b_n)^{k_n} = (1 - a_n)^{k_n} \big(1 + O\big(\frac{b_n}{a_n}\log \frac{1}{a_n}\big)\big)$.
\end{lemma}
\begin{proof}
We note that
\begin{align*}
    \log \frac{(1 - a_n + b_n)^{k_n}}{(1 - a_n)^{k_n}} &= k_n \log \bigg(1 + \frac{b_n}{1 - a_n}\bigg) \\
    &=O\bigg(\frac{1}{a_n} \log \bigg(\frac{1}{a_n} \bigg) (-b_n + o(b_n))\bigg) \\
    &= O\bigg(\frac{b_n}{a_n}\log \frac{1}{a_n}\bigg).
\end{align*}
Exponentiating both sides we get the desired result.
\end{proof}

\begin{corollary}\label{cor:asym}
    We have for $k = O(n^s \log n)$ that $(1 - 2\theta + \psi)^k = (1 - \theta)^{2k}(1 + O(\frac{\log n}{n^{2 \wedge s}})).$
\end{corollary}
\begin{proof}
    Since $\psi = \Theta(n^{-s - 2}) \ll \theta = \Theta(n^{-s})$, we have $(1 - 2\theta + \psi)^k = (1 - 2\theta)^k(1 + O(\frac{\log n}{n^2}))$ by taking $a_n = 2\theta$ and $b_n = \psi$ in Lemma~\ref{lem:asym}. On the other hand, by taking $a_n = 2\theta$ and $b_n = \theta^2$ in Lemma~\ref{lem:asym}, we have that $(1 - \theta)^{2k} = (1 - 2\theta)^k (1 + O(\frac{\log n}{n^s}))$.
\end{proof}
We are now ready to prove our main results.
\begin{proof}[\textbf{Proof of Theorem~\ref{thm:gccp}}]
Set $M = \big\lceil \frac{n^s \log n}{(s - 1)! \binom{r}{s}} \big\rceil$. We then have
\begin{align*}
    \E T^{(r, s)} &= \sum_{k \ge 0} \P(T^{(r, s)} > k) \\
           &= \sum_{k \ge 0} \P(N_k > 0) \\
           &= \sum_{0 \le k < M} \P(N_k > 0) + \sum_{k \ge M} \P(N_k > 0).
\end{align*}
\textbf{Upper bound:}
We bound the summands in the second sum using the first moment method and \eqref{eq:mom_1}:
\begin{align*}
    \P(N_k > 0) \le \E N_k = m (1 - \theta)^k,
\end{align*}
where $\theta = \frac{\binom{n - s}{r - s}}{\binom{n}{r}} = \frac{s! \binom{r}{s}}{n^s}(1 + o(1))$. Therefore
\begin{align}
    \sum_{k \ge M} \P(N_k > 0) &\le \sum_{k \ge M} m (1 - \theta)^k \nonumber\\
                               &= \frac{m (1 - \theta)^M}{\theta} \nonumber\\
                               &\le \frac{\binom{n}{s} O(n^{-s})}{\frac{s! \binom{r}{s}}{n^s}(1 + o(1))} \nonumber\\
                               &= O(n^s), \label{eq:st}
\end{align}
where in the third line we have used the fact that
\begin{equation}\label{eq:lot}
    (1 - \theta)^M = \bigg(1 - \frac{s! \binom{r}{s}}{n^s}(1 + o(1))\bigg)^{\frac{n^s \log n}{(s - 1)! \binom{r}{s}}} = O(e^{-s \log n}) = O(n^{-s}).
\end{equation}
From \eqref{eq:st} and \eqref{eq:lot} we have
\[
    \E T^{(k, s)} = \sum_{0 \le k < M} \P(N_k > 0) + O(n^s) \le \frac{n^s \log n }{(s - 1)! \binom{r}{s}} + O(n^s),
\]
where the second inequality follows trivially by upper-bounding $\P(N_k > 0)$ by $1$.

\noindent
\textbf{Lower bound:} Since
\begin{align*}
    \E T^{(r, s)} &\geq \sum_{0 \le k < M} \P(N_k > 0),
\end{align*}
it suffices to lower bound $\sum_{0 \le k < M} \P(N_k > 0)$. Note that $N_k \ge N_{k + 1}$, so that $\P(N_k > 0)$ is decreasing in $k$. Therefore, with $0 < \delta_n < \frac{1}{s!\binom{r}{s}}$, and $k_n = \lfloor \big(\frac{1}{(s - 1)! \binom{r}{s}} - \delta_n\big) n^s \log n \rfloor$, we have
\begin{equation}\label{eq:fstbound}
    \sum_{0 \le k < M} \P(N_k > 0) \ge \sum_{0 \le k \le k_n} \P(N_k > 0) \ge \bigg(\frac{1}{(s - 1)! \binom{r}{s}} - \delta_n\bigg) n^s \log n \, \P(N_{k_n} > 0).
\end{equation}
We will suitably choose $\delta_n \rightarrow 0$ later. We now use the second moment method to lower bound $\P(N_{k_n} > 0)$:
\[
    \P(N_{k_n} > 0) \ge \frac{(\E N_{k_n})^2}{\E N_{k_n}^2} = \frac{(\E N_{k_n})^2}{\var(N_{k_n}) + (\E(N_{k_n}))^2} = \frac{1}{1 + \frac{\var(N_{k_n})}{(\E N_{k_n})^2}}.
\]
Let $\gamma_n = (1 - \theta)^{k_n}$. Using the upper bound in \eqref{eq:mom_d} for $d = 2$ and Corollary~\ref{cor:asym}, we get that
\begin{align*}
     \frac{\var(N_{k_n})}{(\E N_{k_n})^2} &\le \frac{4 m(1 - \theta)^{k_n} + m(m - 1) (1 - 2\theta + \psi)^{k_n} - m^2(1 - \theta)^{2{k_n}}}{m^2(1 - \theta)^{2{k_n}}} \\
     &= \frac{4 m \gamma_n + m(m - 1) \gamma_n^2 - m^2 \gamma_n^2 + m^2\gamma_n^2 O(\frac{\log n}{n^{2 \wedge s}})}{m^2 \gamma_n^2} \\
     &= \frac{4 - \gamma_n}{m\gamma_n} + O\bigg(\frac{\log n}{n^{2 \wedge s}}\bigg)\\
     &=O\bigg(\frac{1}{m\gamma_n}\bigg),
\end{align*}
where in the last line we are using the fact that $\delta_n \to 0$ so that
\[
    \frac{1}{m\gamma_n} = \Theta(n^{-s!\binom{r}{s}\delta_n}) = n^{-o(1)} \gg \frac{\log n}{n^{2 \wedge s}}.
\]
Therefore
\begin{align*}
     \P(N_{k_n} > 0) &\ge \frac{1}{1+ O(\frac{1}{m\gamma_n})} \\
     &= 1 - O(n^{-s!\binom{r}{s}\delta_n}).
\end{align*}
Substituting the above lower bound in \eqref{eq:fstbound}, we get that
\[
    \sum_{0 \le k < M} \P(N_k > 0) \ge \bigg(\frac{1}{(s - 1)!\binom{r}{s}} - \delta_n\bigg) n^s \log n - O(n^{s - s!\binom{r}{s} \delta_n} \log n).
\]
We now choose $\delta_n = \frac{\log \log n}{s!\binom{r}{s}\log n}$, for which $n^{s!\binom{r}{s}\delta_n} = \log n$, so that
\begin{align*}
    \sum_{0 \le k < M} \P(N_k > 0) &\ge \bigg(\frac{1}{(s - 1)!\binom{r}{s}} - \delta_n\bigg) n^s \log n - O(n^{s}) \\
    &= \frac{n^s \log n }{(s - 1)! \binom{r}{s}} + O(n^s \log \log n).
\end{align*}
Combining the upper and the lower bounds we get that
\[
    \E T^{(k, s)} = \frac{n^s \log n }{(s - 1)! \binom{r}{s}} + O(n^s \log \log n).
\]
This completes the proof.
\end{proof}

We now present the proof of Theorem~\ref{thm:alpha-gccp}, which uses higher order moments.
\begin{proof}[\textbf{Proof of Theorem~\ref{thm:alpha-gccp}}]
We note that for any fixed $d \ge 1$ and $k = O(n^s)$, we have from Lemmas~\ref{lem:mom} and \ref{lem:asym} that
\begin{equation}\label{eq:mom_d_asymp}
    \E N_k^{d} = d ! \binom{m}{d} (1 - d \theta)^k (1 + o(1)).
\end{equation}
Now
\begin{align*}
    \E T^{(r, s)}_{\alpha} &= \sum_{k \ge 0} \P(T^{(r, s)}_{\alpha} > k) \\
    &= \sum_{k \ge 0} \P(N_k > \alpha m).
\end{align*}

\noindent
\textbf{Upper bound:}
Suppose $\varrho > 0$ and $M = \varrho n^s$. Using Markov's inequality, we have the estimate
\begin{align*}
    \sum_{k \ge M} \P(N_k > \alpha m) &\le (\alpha m)^{-d} \sum_{k \ge M} \E N_k^d \\ &= (\alpha m)^{-d} \bigg[\frac{d! \binom{n}{d}(1 - d\theta)^M}{d\theta} (1 + o(1)) \bigg] \\
    &= \frac{1}{d\alpha^d A_{r, s}} \exp(-d A_{r, s}\varrho) n^s (1 + o(1)),
\end{align*}
where $A_{r, s} = s! \binom{r}{s}$. Therefore
\[
    \E T^{(r, s)}_{\alpha} \le \bigg[\underbrace{\varrho + \frac{1}{d\alpha^d A_{r, s}} \exp(-d A_{r, s}\varrho)}_{=: G_d(\varrho)}\bigg] n^s (1 + o(1)).
\]
Note that $G_{d}(\varrho)$ is minimised at $\varrho^* = \frac{1}{A_{r, s}}\log(1 / \alpha)$ and the minimum value is
\[
    G_d(\varrho^*) = \frac{1}{A_{r, s}} \log\bigg(\frac{1}{\alpha}\bigg) + \frac{1}{d A_{r, s}}.
\]
Therefore
\begin{equation}\label{eq:alpha-gccp-ubd}
    \E T^{(r, s)}_{\alpha} \le \bigg[\frac{1}{A_{r, s}} \log\bigg(\frac{1}{\alpha}\bigg) + \frac{1}{d A_{r, s}}\bigg] n^s (1 + o(1)).
\end{equation}

\noindent
\textbf{Lower bound:}
We will now prove a lower bound on $\E T^{(r, s)}_{\alpha}$ for which we will use the general Paley-Zygmund inequality: If $Z \ge 0$ is a random variable with finite variance and $0 \le \epsilon \le 1$, then
\[
    \P(Z > \epsilon \E Z) \ge (1 - \epsilon)^2 \frac{(\E Z)^2}{\E Z^2}.
\]
Using \eqref{eq:mom_d_asymp} together with the Paley-Zygmund inequality, we get that
\begin{align*}
    \P(N_k > \alpha m) &=
    \P(N_k^{d} > \alpha^d m^d) \\
    &= \P\bigg(N_k^d > \frac{\alpha^d m^d}{\E N_k^d}\E N_k^d\bigg) \\
    &\ge \bigg(1 - \frac{\alpha^d m^d}{\E N_k^d}\bigg)^2\frac{(\E N_k^d)^2}{\E N_k^{2d}} \\
    &= \bigg(1 - \frac{\alpha^d}{(1 - d\theta)^k}\bigg)^2 (1 + o(1)),
\end{align*}
where we have used the fact that
\[
    \frac{(\E N_k^d)^2}{\E N_k^{2d}} = \frac{(d! \binom{m}{d} (1 - d\theta)^k)^2}{(2d)! \binom{m}{2d} (1 - 2d \theta)^k}(1 + o(1)) = \frac{(1 - d\theta)^{2k}}{(1 - 2d\theta)^{k}}(1 + o(1)) = 1 + o(1),
\]
where the first equality is a consequence of \eqref{eq:mom_d_asymp} and the last equality follows from an application of Lemma~\ref{lem:asym} (with $a_n = 2\theta$ and $b_n = \theta^2$).
Hence for $M' = \varpi n^s$,
\begin{align*}
    \sum_{1 \le k \le M'} \P(N_k > \alpha m) &\ge M' \bigg(1 - \frac{\alpha^d}{(1 - d\theta)^{M'}}\bigg)^2 (1 + o(1)) \\
    &= \underbrace{\varpi (1 - \alpha^d \exp(d A_{r, s}\varpi))^2}_{=: F_d(\varpi)} n^s (1 + o(1)).
\end{align*}
Clearly, this bound is non-trivial for $0 \le \varpi \le \frac{1}{A_{r, s}} \log(1 / \alpha)$.
Reparameterising $\varpi = \frac{1}{A_{r, s}} (\log(1/\alpha) - u)$, where $u \in [0, \log(1/\alpha)]$, we have
\[
    F_{d}(\varpi) = \frac{1}{A_{r, s}} (\log(1/\alpha) - u) (1 - \exp(-d u))^2 =: H_{d}(u),
\]
which is clearly smaller than $\frac{1}{A_{r, s}}\log(1 / \alpha)$ for any $u \in [0, \log(1/\alpha)]$. Putting $u = \frac{\log d}{d}$, we get
\[
    H_{d}\bigg(\frac{\log d}{d}\bigg) = \frac{1}{A_{r, s}} \bigg[\log\bigg(\frac{1}{\alpha}\bigg) - \frac{\log d}{d}\bigg] \bigg(1 - \frac{1}{d}\bigg)^2.
\]
Thus for any large enough integer $d$ such that $\frac{\log d}{d} \in [0, \log(1/\alpha)]$, we have
\begin{equation}\label{eq:alpha-gccp-lbd}
    \E T^{(r, s)}_{\alpha} \ge \frac{1}{A_{r, s}}\bigg[\log\bigg(\frac{1}{\alpha}\bigg) - \frac{\log d}{d}\bigg] \bigg(1 - \frac{1}{d}\bigg)^2 n^s (1 + o(1)).
\end{equation}
From \eqref{eq:alpha-gccp-ubd} and \eqref{eq:alpha-gccp-lbd}, we conclude (by first letting $n \to \infty$ and then letting $d \to \infty$) that
\[
    \lim_{n \to \infty}\frac{\E T_{\alpha}^{(r, s)}}{n^s} = \frac{1}{A_{r, s}}\log\bigg(\frac{1}{\alpha}\bigg).
\]
This completes the proof.
\end{proof}

\begin{proof}[\textbf{Proof of Theorem~\ref{thm:gumbel}}]
Let $X_n = \frac{T^{(r, s)} - \binom{n}{s} \log \binom{n}{s} / \binom{r}{s}}{\binom{n}{s} / \binom{r}{s}}$. Let $k = \bigg \lfloor \frac{\binom{n}{s}\big(\log \binom{n}{s} + x\big)}{\binom{r}{s}} \bigg \rfloor = \frac{n^s \big(\log \binom{n}{s} + x\big)}{s! \binom{r}{s}} (1 + o(1))$. We then have
\[
    \P(X_n \le x) = \P(T^{(r, s)} \le k) = \P(N_k = 0).
\]
Thus it is enough to show that $N_k \xrightarrow{d} \poi(e^{-x})$. To that end, we will show that the moments $\E N_k^d$, $d \ge 1$, converge to the corresponding moments of a $\poi(e^{-x})$ random variable.

We begin with the expression
\[
    \E N_k^d = \bigg(\sum_{S \in \cS} I_{S}^{(k)}\bigg)^d = \sum_{l = 1}^d \binom{m}{l} C_{d, l} \, \E \bigg[\prod_{j = 1}^l I_{S_j}^{(k)}\bigg]
\]
which was proved in \eqref{eq:prod_formula}. If the indicators $(I_S^{(k)})_{S \in \cS}$ were i.i.d. $\bern(e^{-x}/m)$ variables, then $N_k = \sum_{S \in \cS} I_{S}^{(k)}$ would have a $\bino(m, \frac{e^{-x}}{m})$ distribution which converges weakly to $\poi(e^{-x})$. Also, in that case, we would have
\[
    \E \bigg[\prod_{j = 1}^l I_{S_j}^{(k)}\bigg] = \prod_{j = 1}^l \E I_{S_j}^{(k)} = \frac{e^{-lx}}{m^l}.
\]
Thus, in this hypothetical scenario,
\[
    \E N_k^d \rightarrow \sum_{l = 1}^d \frac{C_{d, l}}{l!} e^{-lx}.
\]
This is the $d$-th moment of a $\poi(e^{-x})$ variable.

In the light of the above computation, we need to show that $\E N_k^d \rightarrow \sum_{l = 1}^d \frac{C_{d, l}}{l!} e^{-lx}$ for all $d \ge 1$. In fact, it suffices to show that
\begin{equation}\label{eq:prod_ind}
    \E \bigg[\prod_{j = 1}^l I_{S_j}^{(k)}\bigg] = \frac{e^{-lx}}{m^l} (1 + o(1)).
\end{equation}
To that end, we recall the estimate in \eqref{eq:prod_ind_bounds}:
\[
    (1 - l\theta)^k \le \E \bigg[\prod_{j = 1}^l I_{S_j}^{(k)}\bigg] \leq \bigg(1 - l \theta + \binom{l}{2}\psi\bigg)^k.
\]
Now, note that
\begin{align*}
    (1 - l\theta)^k &= \bigg(1 - \frac{s!\binom{r}{s}}{n^s}(1 + o(1))\bigg)^k \\
    &= \exp(-l(\log m + x)(1 + o(1))) \\
    &= \frac{e^{-lx}}{m^l} (1 + o(1)).
\end{align*}
Since $\psi = O(n^{-(s + 2)})$, we also have that
\begin{align*}
    \bigg(1 - l \theta + \binom{l}{2}\psi\bigg)^k
    &= (1 - l\theta + O(\psi))^k \\
    &= \exp(-(l(\log m + x) (1 + o(1)) - O(\log n /n^2))) \\
    &= \exp(-l(\log m + x)(1 + o(1))) \\
    &= \frac{e^{-lx}}{m^l} (1 + o(1)).
\end{align*}
This proves \eqref{eq:prod_ind} and we are done.
\end{proof}

We will now prove Theorem~\ref{thm:rw_r=s}. We first recall a couple of results from the theory of random walk on finite graphs. Let $G = (V, E)$ be a finite connected graph. Let $H(u, v)$ be the hitting time of a vertex $v$ for the random walk started at vertex $u$, i.e. the expected number of steps taken by the random walk before vertex $v$ is reached. Also, let $t_{\mathrm{cov}}^{(u)}(G)$ denote the cover time of $G$ (i.e. the expected number of steps to reach every vertex) for a random walk starting at vertex $u$. Then we have the following result (see, e.g., Theorem 2.7 in \cite{lovasz1993random}).
\begin{theorem}[Matthews' theorem]
    For any $w \in V$, we have
\[
    \bigg(1 + \frac{1}{2} + \cdots + \frac{1}{|V|}\bigg) \min_{u, v} H(u, v) \le t_{\mathrm{cov}}^{(w)}(G) \le \bigg(1 + \frac{1}{2} + \cdots + \frac{1}{|V|}\bigg) \max_{u, v} H(u, v).
\]
\end{theorem}

We also need a recursive formula for the hitting times $H(u, v)$ (which can be obtained by a first-step analysis of the random walk, see, e.g., \cite[pp. 15]{lovasz1993random}).
\begin{proposition}[Recursion of hitting times]
For any $u, v \in  V$
\begin{equation}\label{eq:hittm}
    H(u, v) = 1 + \frac{1}{d_u}\sum_{w \,:\, (u, w) \in E} H(w, v),
\end{equation}
where $d_u$ is the degree of vertex $u$.
\end{proposition}

\begin{proof}[\textbf{Proof of Theorem~\ref{thm:rw_r=s}}]
Consider the graph $G$ each of whose vertices is a super-coupon of size $r$ and two vertices are connected by an edge if they share all but a single coupon. Thus $G$ is an $r(n - r)$-regular graph on $m = \binom{n}{r}$ vertices. It is not hard to see that $\E T_{\mathrm{RW}}^{(r, r)} - 1$ is precisely the cover time (i.e. the expected time to visit all the vertices) of a random walk on $G$ started at a uniformly random vertex. Now, by symmetry, $t_{\mathrm{cov}}^{(u)}(G)$ does not depend on $u$. Denoting the common value by $t_{\mathrm{cov}}(G)$, we have that $\E T_{\mathrm{RW}}^{(r, r)} = 1 + t_{\mathrm{cov}}(G)$.

    We first observe that $G$ possesses the following symmetry: given any two pairs of super-coupons $u, v$ and $u', v'$ such that $|u\cap v| = |u'\cap v'|$, there is an automorphism which sends $u$ and $v$ to $u'$ and $v'$ respectively. Thus the hitting time $H(u, v)$ is only dependent on $|u \cap v|$. Let $h_k$ denote the hitting time $H(u, v)$ for some $u, v$ such that $|u \cap v| = k$.

We now observe that if $|u \cap v| = k$ and $w$ is a neighbour of $u$, i.e. $|u \cap w| = r - 1$, then $|w \cap v|$ can take at most three values, namely $k - 1$, $k$, or $k + 1$. Indeed,
\[
    |w \cap v| = |u \cap w \cap v| + |u^c \cap w \cap v| \le |u \cap w \cap v| + |u^c \cap w| \le k + 1,
\]
and
\[
    |w \cap v| \ge |u \cap w \cap v| = |u \cap v| - |u \cap v \cap w^c| \ge |u \cap v| - |u \cap w^c| \ge k - 1.
\]
    
    The only way to construct a neighbour $w$ of $u$ such that $|w \cap v| = k - 1$ is to delete one coupon from $u$ which is in $u \cap v$ and add a coupon not in $u \cup v$. Thus there are $k (n - 2r + k)$ many such neighbours $w$ of $u$.

    On the other hand, the only way to construct a neighbour $w$ of $u$ such that $|w \cap v| = k + 1$ is to delete one coupon from $u$ which is not in $u \cap v$ and add a coupon from $v$ which is not in $u \cap v$. Thus there are $(r - k)^2$ many such neighbours $w$ of $u$.

    All the other neighbours $w$ of $u$ satisfy $|w \cap v| = k$.

Therefore from \eqref{eq:hittm}, we get that for $0 \le k \le r - 1$,
\begin{equation}\label{eq:rec-h_k}
    h_k = 1 + \frac{k(n-2r+k)}{r(n-r)} h_{k-1} + \frac{(r-k)^2}{r(n-r)} h_{k+1} + \bigg[1 - \frac{k(n-2r+k) + (r-k)^2}{r(n-r)}\bigg]h_k,
\end{equation}
where we use the convention that $h_{-1} = h_r = 0$. Defining $x_k := h_k - h_{k - 1}$, $0 \le k \le r - 1$, we obtain the following recursion from \eqref{eq:rec-h_k}:
\[
    k(n -2r + k) x_k = r(n - r) + (r - k)^2 x_{k + 1},
\]
which yields
\begin{equation}\label{eq:rec-x_k}
    x_{k + 1} = \bigg[1 + \frac{nk - r^2}{(r - k)^2}\bigg] x_k - \frac{r(n - r)}{(r - k)^2}.
\end{equation}
Putting $k = 0$ in  \eqref{eq:rec-x_k}, we get
\begin{equation}
    x_1 = -\frac{r(n-r)}{r^2} = -\frac{n}{r}(1+o(1)).
\end{equation}
An induction on $k$ now gives that for $0 \le k \le r - 1$,
\begin{equation}\label{eq:x-sol}
    x_{k + 1} = - \frac{r k!n^{k + 1}}{\prod_{i = 0}^k (r - i)^2} (1 + o(1)).
\end{equation}
In particular, for $k = r - 1$, we get
\[
    x_r = -\frac{r (r-1)! n^r}{(r!)^2} (1+o(1)) = -\frac{n^r}{r!} (1+o(1)).
\]
As $h_r = 0$ and $x_r = h_r - h_{r-1}$, we get that
\[
    h_{r - 1} = \frac{n^r}{r!}(1 + o(1)).
\]
Finally, we note that \eqref{eq:x-sol} implies that for all $1 \le k \le r - 1$, $x_k = o(n^r)$. Then, since
\[
    h_k = h_{r - 1} - x_{r - 1} - x_{r - 2} - \cdots - x_{k + 1},
\]
we have that for all $0 \le k \le r - 1$,
\[
    h_{k} = \frac{n^r}{r!} (1 + o(1)).
\]
Matthews' theorem now implies that the cover time is $\frac{n^r}{r!} \log \binom{n}{r} (1 + o(1)) = \frac{n^r}{(r - 1)!} \log n (1 + o(1))$. This completes the proof.
\end{proof}

\section*{Funding}
The research of SA is partially supported by the CPDA grant from Indian Statistical Institute and the Knowledge Exchange grant from ICTS-TIFR. The research of SM is supported by the German Research Foundation through the grant DFG-ANR PRCI ``ASCAI'' (GH 257/3-1). The research of SSM is partially supported by the INSPIRE research grant DST/INSPIRE/04/2018/002193 and the CPDA grant from Indian Statistical Institute.

\section*{Acknowledgements}
The authors thank Partha Dey and Svante Janson for their helpful comments.

\newpage

\bibliographystyle{alpha-abbrv}
\bibliography{ref}

\end{document}